\documentclass[11pt,leqno]{amsart}
\usepackage{epsfig}
\usepackage{amssymb}
\usepackage{amscd}
\usepackage[matrix,arrow]{xy}
\usepackage{graphicx}

\newtheorem{theo}{Theorem}[section]
\newtheorem{lemma}[theo]{Lemma}
\newtheorem{defi}[theo]{Definition}
\newtheorem{prop}[theo]{Proposition}

\newtheorem{remark}[theo]{Remark}

\numberwithin{equation}{section}

\def\bL{\mathbb{L}}

\def\C{\mathbb{C}}
\def\Z{\mathbb{Z}}
\def\Q{\mathbb{Q}}

\def\bR{{\mathbf R}}
\def\bL{{\mathbf L}}

\def\pre-tr{\operatorname{pre-tr}}

\def\Hom{\operatorname{Hom}}

\newcommand{\Hoch}{\operatorname{Hoch}}

\newcommand{\charact}{\mathrm char}

\newcommand{\cF}{{\mathcal F}}
\newcommand{\cG}{{\mathcal G}}
\newcommand{\cO}{{\mathcal O}}

\newcommand{\cD}{{\mathcal D}}
\newcommand{\cV}{{\mathcal V}}

\newcommand{\cU}{{\mathcal U}}

\newcommand{\cK}{{\mathcal K}}

\newcommand{\veps}{\varepsilon}

\newcommand{\un}{\underline}

\newcommand{\Com}{\operatorname{Com}}

\newcommand{\Perf}{\operatorname{Perf}}

\newcommand{\supp}{\operatorname{Supp}}

\newcommand{\coker}{\operatorname{Coker}}

\newcommand{\Id}{\operatorname{Id}}

\newcommand{\colim}{\operatorname{colim}}

\newcommand{\Pre}{\operatorname{Pre}}

\newcommand{\Spec}{\operatorname{Spec}\,}
\newcommand{\Spf}{\operatorname{Spf}\,}

\newcommand{\Ho}{\operatorname{Ho}}

\newcommand{\id}{\operatorname{id}}

\usepackage{epsf}
\usepackage{amscd}

\title[Cyclic homology of categories of matrix factorizations]
{Cyclic homology of categories of matrix factorizations}

\author{Alexander I. Efimov}
\address{Steklov Mathematical Institute of RAS, Gubkin str. 8, Moscow 119991, Russia}
\email{efimov@mccme.ru}
\thanks{This work is supported by the RSF under a grant 14-50-00005.}

\begin{document}

\begin{abstract}In this paper, we will show that for a smooth quasi-projective variety over $\C,$ and a regular function $W:X\to \C,$ the periodic cyclic homology of the DG category of matrix factorizations $MF(X,W)$ is identified (unde Riemann-Hilbert correspondence) with vanishing cohomology $H^{\bullet}(X^{an},\phi_W\C_X),$ with monodromy twisted by sign.
Also, Hochschild homology is identified respectively with the hypercohomology of $(\Omega_X^{\bullet},dW\wedge).$

One can show that the image of the Chern character is contained in the subspace of Hodge classes. One can formulate the Hodge conjecture stating that it is surjective ($\otimes\Q$) onto Hodge classes. For $W=0$ and $X$ smooth projective this is precisely the classical Hodge conjecture.\end{abstract}


\maketitle

\tableofcontents

\section{Introduction}

It is classically known that Hochschild homology is a non-commutative analog of differential forms, and in characteristic zero Connes differential corresponds to de Rham differential \cite{HKR}, \cite{FT}, \cite{C}. The de Rham cohomology hence corresponds to periodic cyclic homology.

An identification of the de Rham cohomology with periodic cyclic homology for affine smooth schemes was shown by Feigin and Tsygan \cite{FT}, and they also proved that for non-smooth affine schemes the periodic cyclic homology is identified with crystalline cohomology (which in the case of complex numbers is just the singular cohomology of the associated topological space).

Weibel \cite{W2} (see also Weibel and Geller \cite{WG}) has shown that if one defines the cyclic homology of schemes of finite type by sheafifying the cyclic complex and then taking $\bR\Gamma,$ one still gets
the crystalline cohomology.

Keller \cite{Ke2} proved that this version of cyclic homology actually coincides with the cyclic homology of the DG category of perfect complexes on a scheme.
Therefore, the periodic cyclic homology of $\Perf(X)$ is identified with the crystalline cohomology.

In this paper we compute periodic cyclic homology for matrix factorization categories. More precisely, following \cite{KKP}, we consider periodic cyclic homology of a $\Z/2$-graded DG category as a $\Z/2$-graded
vector bundle with connection on the formal punctured disk. We show that given a quasi-projective smooth complex algebraic variety
with a regular function, the periodic cyclic homology is identified under Riemann-Hilbert correspondence with cohomology of the sheaf of vanishing cycles
with the monodromy twisted by a sign. This can be considered as a generalization of a computation of the periodic cyclic homology of perfect complexes.

For a scheme $X$ of finite type over a field $k,$ together with a regular function $W\in\cO(X),$ we denote by $MF_{coh}(X,W)$ the $\Z/2$-graded DG category
of matrix factorizations, defined in \cite{Pos} in a more general context of quasi-coherent sheaves of curved DG algebras.

Further, for any small $\Z/2$-graded DG category $T$ over $k$ we recall the Hochschild homology
$$HH_{\bullet}(T):=H^{-\bullet}(\Id_T\stackrel{\bL}{\otimes}_{T\otimes T^{op}}\Id_T),$$
where $\Id_T$ is the diagonal DG bimodule. It can be computed by the standard Hochschild complex $(\Hoch(T),b),$ obtained from the bar resolution of one of the copies of $\Id_T.$ There is a well-known Connes differential $B$ on $\Hoch(T),$ satisfying $$Bb+bB=0.$$ It allows one to define the periodic cyclic homology
$$HP_{\bullet}(T):=H^{\bullet}(\Hoch((u)),b+uB),$$ as well as ordinary cyclic homology $HC_{\bullet}(T)$ and negative cyclic homology $HC^{-}_{\bullet}(T),$
which we will not consider.

Periodic cyclic homology $HP_{\bullet}(T)$ can be viewed as a vector bundle over the formal punctured disk $\Spf k((u)),$ with a natural connection $\nabla_u^{GM}.$
introduced in \cite{KKP} and written down explicitly in \cite{S}.

Given a finite-dimensional complex vector space $V$ with an automorphism $T:V\to V,$ choose any $M:V\to V$ such that
$$\exp(-2\pi iM)=T,$$
and write
$$\widehat{RH}^{-1}(E,T):=(V((u)),d+M\frac{du}u).$$
This is a bundle with connection on $\Spf \C((u)).$

Our main result is the following:

\begin{theo}\label{intro:main_complex}Let $X$ be a smooth quasi-projective algebraic variety over $\C,$ and $W:X\to\C$ a regular function. Then we have an identification of $\Z/2$-graded vector bundles with connection on the formal punctured disk:
$$(HP_{\bullet}(MF_{coh}(X,W)),\nabla_u^{GM})\cong \widehat{RH}^{-1}(H^{\bullet-1}_{an}(W^{-1}(0),\phi_W\C_X),T\cdot (-1)^{\bullet}),$$
where $T$ is the monodromy automorphism.\end{theo}

There is an algebraic formula for the vanishing cohomology with monodromy, conjectured by Kontsevich and proved by Sabbah \cite{Sab} (see also the paper by Sabbah and M. Saito \cite{SS}). We formulate its special case needed for our purposes.

\begin{theo}\label{th:Sabbah}(\cite{Sab}) Let $X$ be a smooth quasi-projective algebraic variety over $\C,$ and $W:X\to\C$ a regular function.
Assume that the critical locus of $W$ is contained in $W^{-1}(0).$ Then we have an isomorphism of $\Z$-graded bundles with connections on the formal punctured disk:
$$(H^{\bullet}_{Zar}(X,(\Omega_X^{\bullet}((u)),-dW+ud)),\nabla_u=\frac{d}{du}+\frac{W}{u^2})\cong \widehat{RH}^{-1}(H^{\bullet-1}_{an}(W^{-1}(0),\phi_W\C_X),T).$$\end{theo}

Our Theorem \ref{intro:main_complex} is actually obtained as a combination of Sabbah Theorem and the following result valid over arbitrary field of characteristic zero.

\begin{theo}\label{intro:main_algebraic}Let $X$ be a smooth separated scheme of finite type over a field $k$ of characteristic zero, and $W\in\cO(X)$ a regular function. Assume that the critical locus of $W$ is contained in $W^{-1}(0).$ Then we have natural isomorphisms of $\Z/2$-graded bundles with connections
$$(HP_{\bullet}(MF_{coh}(X,W)),\nabla_u^{GM})\cong (H^{\bullet}(\Omega_X^{\bullet}((u)),-dW+ud),\nabla_u^{DR}=\frac{d}{du}+\frac{\Gamma}{u}+\frac{W}{u^2}),$$
where $\Gamma_{|\Omega^p}:=-\frac{p}2\cdot\id,$
and an isomorphism of $\Z/2$-graded vector spaces $$HH_{\bullet}(MF_{coh}(X,W))\cong H^{\bullet}(X,(\Omega_X^{\bullet},-dW)).$$\end{theo}

Here by $dW$ one means the operator of multiplication by $dW.$ We remark that for Hochschild cohomology it was shown in \cite{LP} that
$$HH^{\bullet}(MF_{coh}(X,W))\cong H^{\bullet}(X,(\Lambda^{\bullet}T_X,\iota_{dW})),$$
again under assumption that the critical locus is in the fiber over zero.

Theorem \ref{intro:main_algebraic} is implied by a stronger result on mixed complexes with $u$-connections
(see Section \ref{sec:mixed_connections} for the precise definitions). Namely, a ($2$-periodic) mixed complex is just
 a $\Z/2$-graded vector space $C$ with a pair of anti-commuting differentials $(b,B).$ The morphism of such mixed complexes is defined in the obvious way, and it is said to be a quasi-isomorphism if it induces isomorphisms on $b$-cohomology. A $u$-connection on a mixed complex is defined in the natural way so that it induces
 a connection on $$H^{\bullet}(C((u)),b+uB).$$ There is a natural notion of weak and strict morphisms of mixed complexes with $u$-connections. Also,
 for a sheaf of mixed complexes with $u$-connections on a scheme $X$ of finite type over a field, there is a functor $\bR\Gamma(X,-)$ with values in mixed complexes with $u$-connections.

  We have the following result.

\begin{theo}\label{intro:main_mixed}Assume that the assumptions of Theorem \ref{intro:main_algebraic} are satisfied. Then
we have a chain of quasi-isomorphisms between mixed complexes with $u$-connections:
$$(\Hoch(MF_{coh}(X,W)),b,B,\nabla_u^{GM}),\quad \bR\Gamma(X,(\Omega_X^{\bullet},-dW,d,\nabla_u^{DR})).$$\end{theo}

The special case of this result for an affine space with a polynomial with isolated singularity has already been shown by Shklyarov \cite{S}.

The paper is organized as follows.

In Section \ref{sec:CDG} we recall the DG categories and curved DG categories over a field, and define the DG categories of matrix factorizations.

Section \ref{sec:mixed_connections} is devoted to mixed complexes with $u$-connections, already mentioned above. First, in Subsection \ref{subsec:mixed}
we introduce mixed complexes, morphisms and quasi-isomorphisms between them. Also, we give examples of Hochschild complexes (of the second kind)
of small (curved) DG categories and twisted de Rham mixed complex. Then, in Subsection \ref{subsec:connections} we introduce $u$-connections of mixed complexes, and also the notions of weak and strict morphisms between mixed complexes with $u$-connections. Next, in Subsection \ref{subsec:ex_of_conn}
we recall the $u$-connections on the Hochschild complexes (of the second kind) and for twisted de Rham mixed complex. For Hochschild complexes of the second kind it was not written explicitly before, but it is obtained naturally e.g. by looking at the general formula for $A_{\infty}$-algebras, and adding a term corresponding to the $m_0,$ which corresponds to the curvature. In Subsection \ref{subsec:sheaves_of_mixed_with_conn} we study (pre)sheaves of mixed complexes
with $u$-connections on a scheme of finite type over a field. Such (pre)sheaves arise for example from presheaves of (curved) DG categories. Here we construct in particular the derived global section functor with values in mixed complexes with $u$-connections, which preserves weak quasi-isomorphisms.

In Section \ref{sec:aff_case} we prove Theorem \ref{intro:main_mixed} in the affine case: $X=\Spec A.$ It uses the results of Polishchuk and Positselski on Hochschild complexes of the second kind, and the
Hochschild-Kostant-Rosenberg map \cite{HKR} from Hochschild complex of the second kind of the CDG algebra $(A,W)$ to the twisted de Rham mixed complex.
The special case when $A$ is the polynomial ring, and $W$ has isolated singularity at the origin, was done by Shklyarov \cite{S}.

In Section \ref{sec:gen_case} we prove the general case of Theorem \ref{intro:main_mixed}, using sheafification argument. Actually, the idea of sheafification for Hochschild complexes of DG categories is due to Keller \cite{Ke2}, who used this for perfect complexes. Our argument is very similar to Keller's one,
although the chain of quasi-isomorphisms occures to be much longer. We also show here that Theorem \ref{intro:main_mixed} implies Theorem \ref{intro:main_algebraic}, which in turn implies Theorem \ref{intro:main_complex}.

In Section \ref{sec:concluding} we write down some remarks.

\smallskip

{\noindent{\bf Acknowledgements.}} I am grateful to L. Positselski and V. Vologodsky for useful discussions.

\section{(Curved) DG categories}
\label{sec:CDG}

In this section we recall the notion of (curved) DG categories and (curved) DG functors, and also define DG categories of matrix factorizations.
We fix some basic field $k.$

\subsection{DG categories}
\label{subsec:DGcat}

Our basic reference for DG categories is \cite{Ke1}. The DG quotients were introduced in \cite{Ke4}, \cite{Dr}, and we give the explicit construction from \cite{Dr}.

\begin{defi}1) A $\Z/2$-graded DG category is a category, for which the sets of morphisms
$\Hom(X,Y)$ are $\Z/2$-graded complexes of vector spaces, the composition maps
$$\Hom(Y,Z)\otimes\Hom(X,Y)\to\Hom(X,Z)$$ are morphisms of complexes, and the identity morphisms are closed
of degree zero.

2) A DG functor $F:T\to T^\prime$ between DG categories is a functor such that
the maps $$F(X,Y):\Hom(X,Y)\to\Hom(F(X),F(Y))$$ are morphisms of complexes.\end{defi}

The basic example is the DG category $\Com k$ of complexes of $k$-vector spaces., with Hom's being Hom-complexes.

DG algebra is a DG category with one object. A morphism of DG algebras is a DG functor of the associated DG categories.

To any DG category $T$ we can associate a graded $k$-linear homotopy category $\Ho^{\bullet}(T),$ with $$\Hom_{\Ho^{\bullet}(T)}(X,Y):=H^{\bullet}(\Hom_T(X,Y)).$$

\begin{defi}A DG functor $F:T\to T^\prime$ is called a quasi-equivalence if the induced functor
$$\Ho^{\bullet}(F):\Ho^{\bullet}(T)\to \Ho^{\bullet}(T^\prime)$$
is an equivalence of graded categories.\end{defi}

Recall that the right DG module over a small DG category $T$ is a DG functor
$T^{op}\to \Com k.$ One gets the derived category $D(T)$ of such DG modules, obtained by localizing with respect to quasi-isomorphisms.

There is a more general notion of Morita equivalence.

\begin{defi}A DG functor $F:T\to T^\prime$ between small DG categories is said to be a Morita equivalence if the restriction functor
$$D(T^\prime)\to D(T)$$ is an equivalence.\end{defi}

We will need the notion of a DG quotient for DG categories.

\begin{defi}(\cite{Dr}) Let $S\subset T$ be a full DG subcategory of a small DG category $T.$ Define the Drinfeld DG quotient $T/S$ as follows. This is a DG category with $Ob(T/S)=Ob(T),$ and it is obtained from $T$ by formally adding odd morphisms $h(X):X\to X$ such that $$d(h(X))=\id_X.$$ That is, we have an identification of graded vector spaces
$$\Hom_{T/S}(X,Y)=\Hom_T(X,Y)\oplus\bigoplus\limits_{\substack{n\geq 1,\\Z_1,\dots,Z_n\in S}}\Hom_T(X,Z_1)\otimes\Hom_T(Z_1,Z_2)\otimes\dots\otimes\Hom_T(Z_n,Y)[n],$$
and the differential comes from differentials on $\Hom$'s in $T$ and from the condition $d(h(Z_i))=\id_{Z_i}.$\end{defi}

\subsection{Matrix factorizations}
\label{subsec:MF}

Now we define the DG category of matrix factorizations. Let $X$ be a separated scheme of finite type over $k,$ and
$W\in\cO(X)$ a regular function. Define the following $\Z/2$-graded DG categories.

First, the category $MF^{nv}(X,W)$ has as objects pairs $(E,\delta)$ of a $\Z/2$-graded locally free sheaf $E=E^0\oplus E^1$ on $X$ and an odd map $\delta:E\to E$ such that
$$\delta^2=W\cdot\id_E.$$ The complexes of morphisms are defined by taking standard $\Hom$-complexes (the differential squares to zero).

Second, the category $MF_{coh}^{nv}(X,W)$ is defined in the same way but the objects are pairs $(E,\delta),$ where $E$ is a
$\Z/2$-graded coherent sheaf on $X.$ In particular, we have a fully faithful DG functor $MF^{nv}(X,W)\to MF_{coh}^{nv}(X,W).$

Further, the full DG subcategory $MF_{coh}^{ex}(X,W)\subset MF_{coh}(X,W)$ is defined to consist of convolutions of exact triples of coherent matrix factorizations. Following \cite{Pos}, \cite{Or2}, \cite{PV}, \cite{LP}, we define the category of coherent matrix factorizations as a DG quotient
$$MF_{coh}(X,W):=MF_{coh}^{nv}(X,W)/MF_{coh}^{ex}.$$
In particular, we have a natural DG functor
\begin{equation}\label{eq:nv_to_quotient}MF^{nv}(X,W)\to MF_{coh}(X,W).\end{equation}

\begin{prop}\label{prop:equiv_for_affine} (\cite{Pos}) Let $X$ be affine and smooth over $k.$ Then the DG functor \eqref{eq:nv_to_quotient} is a quasi-equivalence.\end{prop}

\begin{theo} (\cite{Or1},\cite{Pos}) Assume that $X$ is smooth, and $W$ is non-zero on each connected component. Then there is an equivalence of triangulated categories
$$\Ho(MF_{coh}(X,W))\cong D_{sg}(W^{-1}(0)),$$
where $D_{sg}(Y):=D^b_{coh}(Y)/\Perf(Y).$\end{theo}

It can be easily seen in a number of ways that for any open subset $U\subset X$ containing $W^{-1}(0)$ the restriction functor
$$MF_{coh}(X,W)\to MF_{coh}(U,W)$$
is a quasi-equivalence. In particular, we may (and will if it is convenient) assume that the critical locus of $W$ is contained in $W^{-1}(0).$

\subsection{Curved DG categories}
\label{subsec:CDG}

Our basic reference for curved DG categories is \cite{PP}.

\begin{defi}A $\Z/2$-graded curved DG category (CDG category) is a category $\cD,$ for which the sets of morphisms are $\Z/2$-graded vector spaces equipped with odd maps $$d:\Hom(X,Y)\to \Hom(X,Y)$$
(the "differentials"), and for each object there is a distinguished even morphism $h_X:X\to X$ (a curvature). This data is required to satisfy the following properties:

1) We have $$d(fg)=(df)g+(-1)^{|f|}fdg$$ for any composable homogeneous morphisms $f,g;$

2) We have $$d^2(f)=h_Yf-fh_X$$ for any $f:X\to Y;$

3) We have $$dh_X=0$$ for any object $X;$

3) The identity morphisms $\id_X$ are of degree zero (this automatically implies $d\id_X=0$).\end{defi}

CDG algebra is a CDG category with one object.

In particular, any DG category can be treated as a CDG category with $h_X=0.$ Moreover, for any CDG category, its full sub(CDG)category consisting of objects with zero curvature is actually a DG category.

There is a natural notion of an opposite CDG category $\cD^{op}.$ Namely, $$Ob(\cD^{op})=OB(\cD),\quad\Hom_{C^{op}}(X^{op},Y^{op})=\Hom_C(Y,X),\quad d^{op}=d,\quad h_{X^{op}}=-h_X,$$
and the composition changes by a sign: $f^{op}g^{op}=(-1)^{|f||g|}(gf)^{op}.$

We have the following basic example: the CDG category $\Pre(k\text{-mod}).$ Its objects are $\Z/2$-graded vector spaces $X$ equipped with an odd map
$d_X:X\to X.$ Further, $\Hom((X,d_X),(Y,d_Y)).$ is the graded space of morphisms. It is equipped with a "differential"
$$d(f)=d_Yf-(-1)^{|f|}fd_X.$$ The composition is obvious, and the curvature of $(X,d_X)$ equals $d_X^2.$

We have a notion of a CDG functor

\begin{defi}\label{def:CDG}A curved DG functor $(F,\alpha):\cD\to\cD^\prime$ (CDG functor) between CDG categories is a functor together with distinguished odd morphisms $$\alpha_X:F(X)\to F(X)$$ such that

1) The maps $$F(X,Y):\Hom(X,Y)\to \Hom(F(X),F(Y))$$ are maps of graded vector spaces;

2) We have $$F(df)=d(F(f))+\alpha_YF(f)-(-1)^{|f|}F(f)\alpha_X.$$ for any homogeneous $f\in\Hom_{\cD}(X,Y);$

3) We have $$F(h_X)=h_{F(X)}+d\alpha_X+\alpha_X^2.$$\end{defi}

A morphism of CDG algebras is a CDG functor of the associated CDG categories.

It is easy to see that CDG functors from a small CDG category $\cD$ to any CDG category $\cD^\prime$ form a DG category: the space $\Hom((F,\alpha),(G,\beta))$ is the graded space of morphisms between graded functors, the differential is given by the formula $$d(f)_X=d(f_X)+\beta_Xf-(-1)^{|f|}f\alpha_X$$
for a homogeneous morphism $f:(F,\alpha)\to (G,\beta),$ and the composition is obvious.

We will call a CDG functor strict if the morphisms $\alpha_X$ equal to zero. In particular, any DG functor between DG categories can be treated as a strict CDG functor between associated CDG categories.

\begin{defi}A (left) CDG module over a CDG category $\cD$ is a strict CDG functor $F:\cD\to \Pre(k\text{-mod}).$ Left CDG-modules over $\cD$ form a
DG category $\cD-mod^{cdg}.$\end{defi}

We will also need a notion of a QDG functor.

\begin{defi}A QDG functor $F:\cD\to\cD^\prime$ between CDG categories is given by the same data as a CDG functor, but we do not require the condition 3) of Definition \ref{def:CDG} to hold (the relation with curvatures).

The QDG functors form a CDG category: the morphisms and the "differentials" are defined in the same way as for CDG functors, and the curvatures are given by the formula
$$(h_F)_X:=h_{F(X)}+d\alpha_X+\alpha_X^2-F(h_X).$$\end{defi}

It is clear that the DG category of CDG functors is a full CDG subcategory of the CDG category of QDG functors, consisting of QDG functors with zero curvature.

Again, we call a QDG functor {\it strict,} if $\alpha_X=0.$

\begin{defi}A (left) QDG module over a CDG category $\cD$ is a strict QDG functor $F:\cD\to \Pre(k\text{-mod}).$ Left QDG-modules over $\cD$ form a
CDG category $\cD-mod^{qdg}.$ In particular, we have $$\Pre(k\text{-mod})\cong k\text{-mod}^{qdg}.$$\end{defi}

There is no reasonable notion of a weak equivalence for CDG categories (like e.g. quasi-equivalences or Morita equivalences of DG categories). However, we will need a notion of a pseudo-equivalence \cite{PP}, which is quite useful for our purposes.

First, define the notion of a twist of an object. Let $X\in\cD$ be an object of a CDG category, and $\tau:X\to X$ an odd endomorphism. Then an object $Y\in\cD$
is called a twist of an object $X$ by $\tau$ if there exist even morphisms $i:X\to Y,$ $j:Y\to X$ such that
$$ij=\id_Y,\quad ji=\id_X,\quad jd(i)=\tau,\quad jh_Yi=h_X+d\tau+\tau^2.$$

The shift $X[1]$ of $X$ is defined as an object $Y$ equipped with odd morphisms $i:X\to Y,$ $j:Y\to X$ such that
$$ij=\id_Y,\quad ji=\id_X,\quad d(i)=0,\quad d(j)=0.$$

Further, given a family of objects $X_{\alpha}\in\cD,$ there direct sum is an object $X,$ which is equipped with a structure of a direct sum of objects
$X_{\alpha}$ in the graded category $\cD^{gr},$ and the structural even morphisms $i_{\alpha}:X_{\alpha}\to X$ satisfy $di_{\alpha}=0.$

\begin{defi}A CDG functor $F:\cD\to\cD^\prime$ is called a pseudo-equivalence if it is fully faithful (as a functor between graded categories) and each
object of $\cD^\prime$ can be obtained from the image of $F$ using direct sums, direct summands, shifts and twists.\end{defi}

Take any small CDG category $\cD.$ Then we have full (C)DG  subcategories
$$\cD\text{-mod}^{cdg}_{fgp}\subset \cD\text{-mod}^{cdg},\quad \cD\text{-mod}^{qdg}_{fgp}\subset \cD\text{-mod}^{qdg}$$
consisting of CDG (resp. QDG) modules which are finitely generated projective as graded $\cD$-modules.

Note that we have an Yoneda strict CDG functor
$$\cD^{op}\to \cD\text{-mod}^{qdg}_{fgp}.$$

\begin{prop}\label{prop:cdg_qdg_ps_eq}(\cite{PP})The strict CDG functors
$$\cD\text{-mod}^{cdg}_{fgp}\to \cD\text{-mod}^{qdg}_{fgp}\leftarrow \cD^{op}$$
are pseudo-equivalences.\end{prop}

For any commutative $k$-algebra $A$ of finite type, and any $W\in A,$ we have a CDG algebra $(A,W)$ concentrated in even degree, with zero "differential" and curvature $W.$ Then there is a natural identification of DG categories
$$(A,W)\text{-mod}_{fgp}^{cdg}\cong MF^{nv}(\Spec A,W).$$

\section{Mixed complexes with a connection}
\label{sec:mixed_connections}

We will consider $\Z/2$-graded complexes. Again, we fix some basic field $k.$

\subsection{Mixed complexes.}
\label{subsec:mixed}

\begin{defi}1) A mixed complex is a triple $(C,b,B),$ where $C$ is a $\Z/2$-graded vector space, $b$ and $B$ are odd differentials on $C$ satisfying
$$bB+Bb=0.$$

2) A morphism of mixed complexes
$$f:(C,b,B)\to (C^\prime,b^\prime,B^\prime)$$ is a morphism $f:C\to C^\prime$ of graded vector spaces (homogeneous of degree zero), such that
$$b^\prime f=fb,\quad B^\prime f=fB.$$

3) A morphism $f:(C,b,B)\to (C^\prime,b^\prime,B^\prime)$ is called a quasi-isomorphism if it induces a quasi-isomorphism of complexes
$$(C,b)\to (C^\prime,b^\prime).$$\end{defi}

It is convenient to think of mixed complexes as DG modules over $k\langle B\rangle/B^2,$ where $B$ is an odd variable.

We will consider the following examples of mixed complexes.

{\bf Example:} {\it Hochschild complexes.} Let $T$ be a small $\Z/2$-graded DG category over a field $k.$ Define the DG category $T^e$ by formally adding a (closed) identity morphism $e_X$ for each object $X\in T.$ That is,
$$Ob(T^e):=Ob(T),\quad \Hom_{T^e}(X,Y):=\begin{cases}\Hom_T(X,Y) & \text{if }X\ne Y;\\
\Hom_T(X,X)\oplus k\cdot e_X & \text{if }X=Y,\end{cases}$$
and the composition is obvious.

For a $\Z/2$-graded vector space $V,$ denote by $sV$ the same space with grading shifted by $1$ (reversed grading), and for a homogeneous element $v\in V$ denote by $sv$ the corresponding element of $sV.$
Now, put
\begin{multline*}\Hoch(T):=\bigoplus_{X\in Ob(T)}\Hom_T(X,X)\oplus\\ \bigoplus\limits_{\substack{n\geq 1,\\X_0,\dots,X_n\in Ob(T)}}\Hom_{T^e}(X_n,X_0)\otimes s\Hom_T(X_{n-1},X_n)\dots\otimes s\Hom_T(X_0,X_1).\end{multline*}
We will write $(f_n,f_{n-1},\dots,f_0)$ for $f_n\otimes sf_{n-1}\otimes\dots sf_0.$
The Hochschild differential $b$ is defined by the formula $$b=b_{m_2}+b_{m_1},$$ where
\begin{multline*}b_{m_2}(f_0\otimes f_1\otimes\dots f_n):=\sum\limits_{i=0}^{n-1}(-1)^{\sum\limits_{k=i+1}^n|sf_k|+1}(f_n,\dots,f_{i+1}f_i,\dots f_0)+\\
(-1)^{|sf_0|(|f_n|+\sum\limits_{k=1}^{n-1}|sf_k|)}(f_0f_n,f_{n-1},\dots,f_1),\end{multline*} and
$$b_{m_1}(f_n,f_{n-1},\dots,f_0)=\sum\limits_{i=0}^n (-1)^{\sum\limits_{k=i+1}^n|sf_k|} (f_n,\dots,df_i,\dots,f_0).$$ Further, the Connes differential $B$ is defined by the formula
$$B(f_n,f_{n-1},\dots,f_0):=\begin{cases}\sum\limits_{i=0}^n (-1)^{(\sum\limits_{k=0}^{i-1}|sf_k|)(\sum\limits_{l=i}^n|sf_l|)} (e_{X_i},f_{i-1},\dots f_0,f_n,\dots f_i) & \text{if }f_n\in\Hom_T(X_n,X_0);\\
0 & \text{if }f_n\in k\cdot e_{X_0}.\end{cases}$$

It is straightforward to check that $(\Hoch(T),b,B)$ is a mixed complex.

{\bf Example:} {\it Hochschild complexes of the second kind.} Suppose that $\cD$ is a $\Z/2$-graded curved small DG category (for instance, $\cD$ can be just a DG category, with zero curvatures). Again, we consider the CDG category $\cD^e$ defined in the same way as above. Define the graded vector space $\Hoch^{\Pi}(\cD)$ by the formula
\begin{multline*}\Hoch^{\Pi}(\cD):=\bigoplus_{X\in Ob(\cD)}\Hom_{\cD}(X,X)\oplus \\\prod\limits_{n\geq 1}(\bigoplus\limits_{X_0,\dots,X_n\in Ob(\cD)}\Hom_{\cD^e}(X_n,X_0)\otimes s\Hom_{l\cD}(X_{n-1},X_n)\otimes\dots\otimes s\Hom_{\cD}(X_0,X_1)).\end{multline*}

The Hochschild differential is defined by the formula $$b=b_{m_2}+b_{m_1}+b_{m_0},$$ where $b_{m_1}$ and $b_{m_2}$ are as above, and
$$b_{m_0}(f_n,f_{n-1},\dots f_0):=\sum\limits_{i=0}^{n}(-1)^{\sum\limits_{k=i}^n|sf_k|}(f_n,\dots, f_i,h_{X_i},f_{i-1},\dots,f_0).$$

The Connes differential $B:\Hoch^{\Pi}(\cD)\to \Hoch^{\Pi}(\cD)$ is defined by the same formula as above.

\begin{remark}One can show (by an easy spectral sequence argument) that if we would take direct sums instead of direct products in the definition of $\Hoch^{\Pi}(\cD),$ and at least one curvature $h_X$ is non-zero, then the resulting complex would be acyclic (see \cite{PP}).\end{remark}

{\bf Example:} {\it Twisted de Rham mixed complex.} Let $A$ be a commutative smooth $k$-algebra of finite type. Then $$(\Omega^{\bullet}(A),-dW\wedge,d_{DR})$$ is a mixed complex.

\subsection{$u$-connections on mixed complexes.}
\label{subsec:connections}

Take a formal even variable $u.$ Starting from a mixed complex $(C,b,B),$ we can form a new complex $(C((u)),b+uB),$ where
$$C((u))=\colim_n u^{-n}\cdot C[[u]],\quad C[[u]]=\prod\limits_{m\geq 0}u^m\cdot C.$$

\begin{prop}Any morphism of mixed complexes $f:(C,b,B)\to (C^\prime,b^\prime,B^\prime)$ induces a morphism of complexes
$$f((u)):(C((u)),b+uB)\to (C^\prime((u)),b^\prime+uB^\prime).$$ Moreover, if $f$ is a quasi-isomorphism, then so is $f((u)).$\end{prop}

\begin{proof}The first statement is obvious, and the second follows from the spectral sequence argument.\end{proof}

\begin{defi}A $u$-connection on a mixed complex $(C,b,B)$ is a $k$-linear operator
$$\nabla_u:C((u))\to C((u)),$$ which has the form
$$\nabla_u=\frac{d}{du}+A(u),$$
where $A(u)$ is a $k((u))$-linear operator from $C((u))$ to itself, and, moreover,
$$[\nabla_u,b+uB]=\frac1{2u}(b+uB).$$\end{defi}

We will consider morphisms compatible with $u$-connection.

\begin{defi}A morphism $f:(C,b,B,\nabla_u)\to (C^\prime,b^\prime,B^\prime,\nabla_u^\prime)$ of mixed complexes is

1) strictly compatible with $u$-connections
if
$$\nabla_u^\prime f((u))=f((u))\nabla_u.$$ In this case we will call $f$ a strict morphism;

2) weakly compatible with $u$-connections if the $k((u))$-linear operator
$$(\nabla_u^\prime f((u))-f((u))\nabla_u):C((u))\to C^\prime((u))$$
is $u$-homotopic to zero. That is, there exists a $k((u))$-linear operator
$H:C((u))\to C^\prime((u)),$ such that
$$\nabla_u^\prime f((u))-f((u))\nabla_u=(b^\prime+uB^\prime)H+H(b+uB).$$
In this case we will call a pair $(f,H)$ a weak morphism.\end{defi}

\subsection{Examples of $u$-connections.}
\label{subsec:ex_of_conn}

{\noindent}

{\bf Example:} {\it Hochschild complexes.} Let $T$ be a $\Z/2$-graded DG category. Following Shklyarov, we introduce the following $u$-connection on the mixed complex $\Hoch(T):$
$$\nabla_u:=\frac{d}{du}+\frac{\cU_{m_1}}{u^2}+\frac{\cV_{m_1}+\Gamma}{u},$$
where
$$\Gamma(f_n,\dots,f_0):=\frac{-n}2(f_n,\dots,f_0);$$
$$\cU_{m_1}(f_n,\dots,f_0):=\frac12(-1)^{\sum\limits_{k=1}^n|sf_k|+|f_0|(|f_n|+\sum\limits_{k=1}^{n-1}|sf_k|)}((df_0)f_n,f_{n-1},\dots,f_1);$$
$$\cV_{m_1}(f_n,\dots,f_0):=\begin{cases}\frac12\sum\limits_{0\leq i\leq j\leq n-1}(-1)^{\epsilon_{ij}} (e_{X_{j+1}},f_j,\dots,df_i,\dots,f_0,f_n,\dots,f_{j+1}) & \text{if }f_n\in\Hom_T(X_n,X_0);\\
0 & \text{if }f_n\in k\cdot e_{X_0},\end{cases}$$
where
$$\epsilon_{ij}=\sum\limits_{k=i+1}^n|sf_k|+(\sum\limits_{k=0}^{j}|sf_k|+1)(\sum\limits_{l=j+1}^n|sf_l|)$$

It was checked by Shklyarov (\cite{S}, section C.2) that $\nabla_u$ is indeed a $u$-connection.

\begin{remark}This $u$-connection was actually introduced in \cite{KKP} (without explicit formula).
Namely, consider a deformation $T_t$ of $T$ over $G_m:$ objects, spaces of morphisms and compositions do not change, but the differentials change by the formula
$$d_t=t\cdot d$$ (this deformation is naturally isomorphic to the one in \cite{KKP}, the isomorphism map is multiplication by $t^{-1}$ on morphisms). Then we have a $\Z/2$-graded bundle over $G_m\times \Spf k((u)):$
$$E=H^{\bullet}(\Hoch(T_t)((u)),b_t+uB_t).$$
It is equipped with Getzler-Gauss-Manin \cite{G} connection $\nabla^{GM}$ in the $t$-direction. Note that we have a natural identification of mixed complexes
$$(\Hoch(T_t),b_t,B_t)\to (C,tb,t^{-1}B),\quad (f_n,\dots,f_0)\mapsto t^{-n}(f_n,\dots,f_0).$$ This equips $E$ with a $G_m$-equivariant structure with respect to the action $$\mu\cdot(t,u)=(\mu t,\mu^2 u).$$ Let $\Lambda:E\to E$ be an operator of differentiation by $\frac{d}{d\mu}.$ Then $\Lambda$ has symbol
 $(t\frac{\partial}{\partial t}+2u\frac{\partial}{\partial u})\cdot\id_E.$ Therefore, the differential operator
  $$\frac{\Lambda}{2u}-\nabla^{GM}_{\frac{t}{2u}\frac{\partial}{\partial t}}$$
  has symbol $\frac{\partial}{\partial u}.$ Its restriction to the fiber $t=1$ is precisely the $u$-connection described above.\end{remark}

Below we will always write $\Hoch(T)$ for the mixed Hochschild complex with $u$-connection defined above.

\begin{prop}\label{prop:q_is_for_q_eq}A DG functor $F:T\to T^\prime$ induces a morphism of mixed complexes $F_*:\Hoch(T)\to \Hoch(T^\prime),$ strictly compatible with a
$u$-connections. Moreover, if $F$ is a Morita equivalence, then $F_*$ is a quasi-isomorphism.\end{prop}

\begin{proof}The first statement is obvious, and the second one is due to Keller \cite{Ke3}.\end{proof}

We will need one more result here (due to Keller).

\begin{defi}A sequence of small DG categories and DG functors
$T^\prime\to T\to T^{\prime\prime}$ is said to be exact up to Morita equivalence, if the image of the composition consists of objects homotopic to zero. and the choice of contracting homotopies defines a Morita equivalence
$T/T^\prime\to T^{\prime\prime}$\end{defi}

\begin{theo}\label{th:Keller_exact_seq}(\cite{Ke3}) Given an exact sequence of small DG categories $T^\prime\to T\to T^{\prime\prime},$
one has a natural exact triangle of Hochschild complexes $$\Hoch(T^\prime)\to \Hoch(T)\to \Hoch(T^{\prime\prime})\to\\ \Hoch(T^\prime)[1].$$
\end{theo}

{\bf Example:} {\it Hochschild complexes of the second kind.} Let $\cD$ be a $\Z/2$-graded CDG category. Then we define a $u$-connection on the mixed complex $\Hoch^{\Pi}(\cD)$ by the formula
$$\nabla_u:=\frac{d}{du}+\frac{2\cU_{m_0}+\cU_{m_1}}{u^2}+\frac{2\cV_{m_0}+\cV_{m_1}+\Gamma}{u},$$
where $\Gamma,$ $\cU_{m_1}$ and $\cU_{m_2}$ are as above, and
$$\cU_{m_0}(f_n,\dots,f_0):=(h_{X_0}f_n,f_{n-1},\dots,f_0);$$
$$\cV_{m_0}(f_n,\dots,f_0):=\begin{cases}\sum\limits_{0\leq i\leq j\leq n} (-1)^{\eta_{ij}}(e_{X_{j}},f_{j-1},\dots,h_{X_i},\dots,f_0,f_n,\dots,f_{j}) & \text{if }f_n\in\Hom_T(X_n,X_0);\\
0 & \text{if }f_n\in k\cdot e_{X_0},\end{cases}$$
where
$$\eta_{ij}=\sum\limits_{k=i}^n|sf_k|+(\sum\limits_{k=0}^{j-1}|sf_k|+1)(\sum\limits_{l=j}^n|sf_l|)+1$$

One can check that this indeed defines a $u$-connection (just analogously to the case of DG categories).
 A special case of polynomial algebra with zero differential and some curvature was verified in \cite{S}, section D.3. Below we will write $\Hoch^{\Pi}(\cD)$ for the mixed Hochschild complex of the second kind with $u$-connection defined above.

\begin{prop}\label{prop:q_is_for_ps_eq}(\cite{PP}) A strict CDG functor $F:\cD\to \cD^\prime$ induces a morphism of mixed complexes $F_*:\Hoch^{\Pi}(\cD)\to \Hoch^{\Pi}(\cD^\prime),$ strictly compatible with
$u$-connections. Moreover, if $F$ is a pseudo-equivalence, then $F_*$ is a quasi-isomorphism.\end{prop}

One can also show that non-strict CDG functor induces a morphism of mixed complexes weakly compatible with $u$-connections,
but we do not need it.

\begin{prop}Let $T$ be a $\Z/2$-graded DG category. Then the natural map
of mixed complexes $$\Hoch(T)\to \Hoch^{\Pi}(T)$$ is strictly compatible with $u$-connections.\end{prop}

\begin{proof}This follows immediately from the definitions.\end{proof}

We need the following result, due to Polishchuk and Positselski (they consider just Hochschild complexes without additional data).

\begin{prop}\label{prop:Pol-Pos}(\cite{PP}) Let $X=\Spec A$ be affine smooth algebraic variety over $k,$ where $\charact k=0,$ and $W\in\cO(X)$ Assume that the critical locus of $W$ is in the fiber $W^{-1}(0).$ Then the morphism of mixed complexes with $u$-connections
$$\Hoch(MF^{nv}(X,W))\to \Hoch^{\Pi}(MF^{nv}(X,W))$$
is a quasi-isomorphism.\end{prop}

We will need one more fact.

\begin{prop}\label{prop:opposite}For any small curved CDG category $\cD,$ we have a natural weak quasi-isomorphism of mixed complexes with $u$-connections:
$$(\Phi,H):\Hoch^{\Pi}(\cD)\to \Hoch^{\Pi}(\cD^{op}),b^{op},B^{op}).$$\end{prop}

\begin{proof} The proof is a generalization of the proof of Proposition 3.8 of \cite{S}.

First, the morphism $\Phi$ is given by the formula
$$\Phi(f_n,f_{n-1},\dots,f_0):=(-1)^{n+\sum\limits_{0\leq i<j\leq n-1}|sf_i||sf_j|}(f_n^{op},f_0^{op},\dots,f_{n-1}^{op}).$$
It is easily seen to commute with Hochschild and Connes differentials. Moreover, it is obviously a quasi-isomorphism.

Further, the odd morphism $$H=H_{m_1}+H_{m_0}:\Hoch^{\Pi}(\cD)((u))\to \Hoch^{\Pi}(\cD^{op})((u))$$ is given by the formulas
$$H_{m_1}(f_n,\dots,f_0):=\begin{cases}\frac{(-1)^{\epsilon_i}}{2u^2}\sum\limits_{i=0}^{n-1}(df_i^{op},f_{i+1}^{op},\dots,f_n^{op},f_0^{op},\dots,f_{i-1}^{op}) & \text{if }f_n\in\Hom_{\cD}(X_n,X_0);\\
0 & \text{if }f_n\in k\cdot e_{X_0};\end{cases}$$
$$H_{m_0}(f_n,\dots,f_0):=\begin{cases}\frac{(-1)^{\eta_i}}{u^2}\sum\limits_{i=0}^n (h_{X_i^{op}},f_i^{op},\dots,f_n^{op},f_0^{op},\dots,f_{i-1}^{op}) & \text{if }f_n\in\Hom_{\cD}(X_n,X_0);\\
0 & \text{if }f_n\in k\cdot e_{X_0},\end{cases}$$
where
$$\eta_i=\sum\limits_{0\leq k<l\leq i-1}|sf_k||sf_l|+\sum\limits_{i\leq k<l\leq n}|sf_k||sf_l|,$$
$$\epsilon_i=\eta_i+1\text{ for }0\leq i\leq n-1.$$
It is straightforward to check that the pair $(f,H)$ defines a weak morphism.
\end{proof}

{\bf Example:} {\it Twisted de Rham mixed complex.} If $A$ is a smooth commutative $k$-algebra of finite type, we define a $u$-connection on the mixed complex
$(\Omega^{\bullet}(A),-dW\wedge,d)$ by the formula
$$\nabla_u^{DR}:=\frac{d}{du}+\frac{\Gamma}{u}+\frac{W}{u^2},$$
where $$\Gamma_{\mid\Omega^p(A)}:=-\frac{p}2\cdot\id.$$

\begin{prop}\label{prop:HKR_map}Let $A$ be as above. Then the Hochschild-Kostant-Rosenberg map $$\veps:\Hoch^{\Pi}(A,W)\to (\Omega^{\bullet}(A),-dW\wedge,d),$$ defined by the formula $$\veps(a_0,\dots,a_n):=\frac{1}{n!}a_0da_1\wedge\dots\wedge da_n,$$
is a quasi-isomorphism of mixed complexes, weakly compatible with connections. Moreover, the contracting homotopy is given by the formula
$$H=-\frac{Wd}{2u}\veps.$$\end{prop}

\begin{proof}It was checked in \cite{S}, section D.5, that the pair $(\veps,H)$ defines a weak morphism for the special case of a polynomial algebra $A$ with $W$ having isolated singularity, but the checking is the same in the general case. Further, the fact that $\veps$ is a quasi-isomorphism follows from the classical Hochschild-Kostant-Rosenberg theorem \cite{HKR} and an obvious spectral sequence argument, see \cite{Seg}.\end{proof}

\subsection{Sheaves of mixed complexes with $u$-connections.}
\label{subsec:sheaves_of_mixed_with_conn}

Let $X$ be a scheme of finite type over a field $k.$ For any presheaf of $k$-vector spaces $\cF$ on $X$ we define a presheaf $\cF((u))$ of $k((u))$-vector spaces
by the formula
$$\cF((u))(U):=\cF(U)((u)).$$

\begin{lemma}\label{lemma:functor((u))}If $\cF$ is a sheaf on $X,$ then so is $\cF((u)).$ The functor $\cF\mapsto\cF((u))$ is exact on the category of sheaves of $k$-vector spaces.\end{lemma}

\begin{proof}The first statement is an easy consequence of the fact that any open subset of $X$ is quasi-compact.
Indeed, because of this fact we have to check that for any finite collection of open subsets $U_1,\dots,U_n\subset X$
there is an exact sequence
$$0\to\cF(\bigcup\limits_{i=1}^nU_i)((u))\to\prod\limits_{i=1}^n\cF(U_i)((u))\to \prod\limits_{1\leq i<j\leq n}\cF(U_i\cap U_j)((u)).$$
But this follows from the corresponding exact sequence for the sheaf $\cF.$

The second statement follows from the first one.\end{proof}

\begin{defi}1) A (pre)sheaf of mixed complexes with a $u$-connection on $X$ is a (pre)sheaf on $X$ (with Zariski topology) with values in the category of mixed complexes with $u$-connections and morphisms strictly compatible with $u$-connections.

That is, to define such a (pre)sheaf $(\un{C},\un{b},\un{B},\un{\nabla_u})$ we need to define for each open subset $U\subset X$ a mixed complex with $u$-connection $(\un{C}(U),\un{b}(U),\un{B}(U),\un{\nabla_u}(U)),$ and for any inclusion of open subsets $V\subset U$ a map of mixed complexes strictly compatible with $u$-connections
$$f_{UV}:(\un{C}(U),\un{b}(U),\un{B}(U),\un{\nabla_u}(U))\to (\un{C}(V),\un{b}(V),\un{B}(V),\un{\nabla_u}(V)),$$
which satisfy the axioms of a (pre)sheaf.

2) A strict morphism of (pre)sheaves of mixed complexes with $u$-connections $$f:(\un{C},\un{b},\un{B},\un{\nabla_u})\to (\un{C}^{\prime},\un{b}^{\prime},\un{B}^{\prime},\un{\nabla_u}^{\prime})$$ is an assignment to each open subset $U\subset X$ a morphism $$f(U):(\un{C}(U),\un{b}(U),\un{B}(U),\un{\nabla_u}(U))\to (\un{C}^{\prime}(U),\un{b}^{\prime}(U),\un{B}^{\prime}(U),\un{\nabla_u}^{\prime}(U))$$
of mixed complexes strictly compatible with $u$-connections, such that for any $V\subset U$ we get a strictly commutative square.

3) A weak morphism of (pre)sheaves of mixed complexes with $u$-connections
$$(f,H):(\un{C},\un{b},\un{B},\un{\nabla_u})\to (\un{C}^{\prime},\un{b}^{\prime},\un{B}^{\prime},\un{\nabla_u}^{\prime})$$
is a pair of a morphism $$f:(\un{C},\un{b},\un{B})\to (\un{C}^{\prime},\un{b}^{\prime},\un{B}^{\prime})$$
of (pre)sheaves of mixed complexes and an odd $k((u))$-linear morphism
$$H:\un{C}((u))\to\un{C}^\prime((u))$$ of $\Z/2$-graded (pre)sheaves of $k((u))$-vector spaces, such that
for any open subset $U\subset X$ we have
$$\un{\nabla_u}^\prime(U)f(U)((u))-f(U)((u))\un{\nabla_u}(U)=(\un{b}^\prime(U)+u\un{B}^\prime(U))H+H(\un{b}(U)+u\un{B}(U)).$$
In the case of sheaves, such a morphism is called a quasi-isomorphism if the morphism
$$f:(\un{C},\un{b})\to (\un{C}^\prime,\un{b}^\prime)$$ is a quasi-isomorphism of complexes of (pre)sheaves of vector spaces.
\end{defi}

Note that a strict morphism $f$ can be treated as a weak morphism $(f,0).$

Clearly, the standard sheafification construction works for mixed complexes with $u$-connections. We will
denote by $\un{C}^{sh}$ the sheafification of $\un{C}.$

We have an obvious sufficient condition for a morphism of mixed complexes with $u$-connections to be a quasi-isomorphism.

\begin{prop}\label{prop:suffices_for_affine}If a weak morphism of presheaves
$$(f,H):(\un{C},\un{b},\un{B},\un{\nabla_u})\to (\un{C}^{\prime},\un{b}^{\prime},\un{B}^{\prime},\un{\nabla_u}^{\prime})$$
induces quasi-isomorphisms
$$f(U):(\un{C}(U),\un{b}(U))\to (\un{C}^\prime(U),\un{b}^\prime(U))$$
for each open affine subset $U\subset X,$ then $$(f^{sh},H^{sh}):(\un{C}^{sh},\un{b},\un{B},\un{\nabla_u})\to (\un{C}^{\prime sh},\un{b}^{\prime},\un{B}^{\prime},\un{\nabla_u}^{\prime})$$ is a quasi-isomorphism.\end{prop}

\begin{proof}Indeed, open affine subsets of $X$ form a base of Zariski topology, which implies the result.\end{proof}

We will be mostly interested in sheaves of mixed complexes coming from presheaves of (curved) DG categories.

\begin{defi}1) A presheaf $\un{T}$ of (curved) DG categories on $X$ is an assignment to each open subset $U\subset X$ of a small $\Z/2$-graded (curved) DG category $\un{T}(U),$ and to each inclusion $V\subset U$ a DG (resp. strict CDG) functor
$F_{UV}:\un{T}(U)\to \un{T}(V)$ such that for $U_1\supset U_2\supset U_3$ we have an equality of (C)DG functors
$F_{U_1U_3}=F_{U_2U_3}F_{U_1U_2}.$

2) A morphism of presheaves of (C)DG categories $F:\un{T}\to \un{T}^\prime$ is a collection of DG (resp. strict CDG) functors $F(U):\un{T}(U)\to\un{T}^\prime(U)$ such that for any inclusion $V\subset U$ we get a strictly commutative square.\end{defi}

Starting from a presheaf of (C)DG categories $\un{T}$ on $X,$ we obtain a sheaf of mixed complexes with $u$-connections $\Hoch(\un{T})$ (resp. $\Hoch^{\Pi}(T)$) associated with a presheaf
$$U\mapsto \Hoch(\un{T}(U))\text{ (resp. }U\mapsto \Hoch^{\Pi}(T(U))).$$

\begin{prop}1) Let $F:\un{T}\to\un{T}^\prime$ be a morphism of presheaves of DG categories on $X.$ Then it induces a strict morphism of sheaves of mixed complexes with $u$-connections
$$F_*:\Hoch(\un{T})\to \Hoch(\un{T}^\prime).$$ Moreover, if for any open affine subset $U\subset X$ the DG functor
$F(U)$ is a quasi-equivalence, then $F_*$ is a quasi-isomorphism.

2) Let $G:\un{D}\to\un{D}^\prime$ be a morphism of presheaves of CDG categories on $X.$ Then it induces a strict morphism of sheaves of mixed complexes with $u$-connections
$$G_*:\Hoch^{\Pi}(\un{\cD})\to \Hoch^{\Pi}(\un{\cD}^\prime).$$ Moreover, if for any open affine subset $U\subset X$ the CDG functor
$G(U)$ is a pseudo-equivalence, then $G_*$ is a quasi-isomorphism.

3) For any presheaf of DG categories $\un{T}$ on $X$ we have a strict morphism of sheaves of mixed complexes with $u$-connections
$$\Hoch(\un{T})\to \Hoch^{\Pi}(\un{T}).$$\end{prop}

\begin{proof}The strict morphisms are defined in the obvious way, and the quasi-isomorphisms follow from Proposition \ref{prop:suffices_for_affine},
Proposition \ref{prop:q_is_for_q_eq} and Proposition \ref{prop:q_is_for_ps_eq}.\end{proof}

We will need also a sheaf of de Rham mixed complexes with $u$-connections.

\begin{defi}Assume that $X$ is smooth and $\charact k=0,$ and $W\in\cO(X)$ is a regular function. Then we have a (twisted de Rham)
sheaf
$$(\Omega^{\bullet}_X,-dW\wedge,d,\nabla_u^{DR}=\frac{d}{du}+\frac{\Gamma}{u}+\frac{W}{u^2})$$
of mixed complexes with $u$-connections, where $\Gamma_{|\Omega^p}=\frac{-p}2\cdot\id$ is as above.\end{defi}

\begin{prop}\label{prop:HKR_gen}Let $X$ be smooth and $\charact k=0.$ Suppose that we have a function $W\in\cO(X),$ and
consider the presheaf $(\cO_X,W)$ of CDG algebras.

Then the maps $\veps$ and $H$ from Proposition \ref{prop:HKR_map} define a weak morphism of mixed complexes with $u$-connections $$(\veps,H):\Hoch^{\Pi}(\cO_X,W),\un{b},\un{B},\un{\nabla_u})\to (\Omega^{\bullet}_X,-dW\wedge,d,\nabla_u=\frac{d}{du}+\frac{\Gamma}{u}+\frac{W}{u^2}),$$ which is a quasi-isomorphism.\end{prop}

\begin{proof}Indeed, this follows from Proposition \ref{prop:suffices_for_affine} and Proposition \ref{prop:HKR_map}.\end{proof}

Finally we would like to define the derived functor of global sections for sheaves of mixed complexes with $u$-connections. For any sheaf $\cF$ of $k$-vector spaces on $X$ we have a flasque Godement resolution
$$\cF\to G(\cF),$$ where
$$G^0(\cF)(U)=Gode(\cF)(U):=\prod\limits_{x\in U}\cF_x,$$
and $$G^i:=Gode(\coker(G^{i-2}(\cF)\to G^{i-1}(\cF))),\quad i>0,$$
where we put $G^{-1}(\cF):=\cF.$

Suppose that the scheme $X$ has dimension $d.$ Then we can apply the canonical truncation
and get the truncated Godement resolution by acyclic sheaves
$$G_{\leq d}(\cF):=\tau_{\leq d}G(\cF).$$

Note that both Godement and truncated Godement resolutions are functorial. This allows us to apply them for mixed complexes with $u$-connections.

\begin{defi}Let $X$ be a separated scheme of finite type over $k,$ and $\dim X=d.$ Take any sheaf of mixed complexes with connections $\un{C}=(\un{C},\un{b},\un{B},\un{\nabla_u}).$ Applying to it truncated Godement resolution and taking the total complex, we get a new sheaf of mixed complex with $u$-connections
$G_{\leq d}(\un{C}).$
We put
$$\bR\Gamma(X,\un{C}):=\Gamma(X,G_{\leq d}(\un{C})).$$
This is a mixed complex with $u$-connection.\end{defi}

Clearly, $\bR\Gamma$ is functorial with respect to weak and strict morphisms. We need the following basic result.

\begin{prop}\label{prop:RGamma_pres_q_is}The functor $\bR\Gamma$ preserves weak quasi-isomorphisms.\end{prop}

\begin{proof}It suffices to prove that for any acyclic $\Z/2$-graded complex of sheaves of vector spaces $\un{C}$ the complex $\Gamma(X,G_{\leq d}(\un{C}))$ is acyclic.

To show this, note that by induction the complex $G^n(\un{C})$ is acyclic for all $n\geq 0.$ Hence, the complex $G_{\leq d}(\un{C})$ is acyclic. But it consists of acyclic sheaves. This implies that $\Gamma(X,G_{\leq d}(\un{C}))$ is also acyclic.\end{proof}

\begin{prop}\label{prop:mixed_implies_k((u))}Let $(\un{C},\un{b},\un{B},\un{\nabla_u})$ be a sheaf of mixed complexes with $u$-connections. Then
we have an isomorphism of bundles with connection on the formal punctured disk:
$$(H^{\bullet}(\bR\Gamma(X,(\un{C})((u)),b+uB),\nabla_u))\cong (H^{\bullet}(X,(\un{C}((u)),\un{b}+u\un{B})),\nabla_u).$$
In particular, if we have a chain of quasi-isomorphisms between mixed complexes
$$(C^\prime,b^\prime,B^\prime,\nabla_u^\prime),\quad \bR\Gamma(X,(\un{C},\un{b},\un{B},\un{\nabla_u})),$$
then we get an isomorphism
$$(H^{\bullet}(C^\prime((u)),b^\prime+uB^\prime),\nabla_u^\prime)\cong (H^{\bullet}(X,(\un{C}((u)),\un{b}+u\un{B})),\nabla_u).$$\end{prop}

\begin{proof}For any sheaf $\cF$ of $k$-vector spaces on $X,$ we claim that $G_{\leq d}(\cF)((u))$ is an acyclic resolution of the sheaf $\cF((u)).$
Indeed, the fact that it is a resolution follows from the exactness of $-((u))$ (Lemma \ref{lemma:functor((u))}), so we have to check that it consists of acyclic sheaves. This reduces to the following lemma.

\begin{lemma}For any acyclic sheaf $\cG$ on $X,$ the sheaf $\cG((u))$ is also acyclic.\end{lemma}

\begin{proof}In the special case when $\cG$ is flasque, we have that $\cG((u))$ is also flasque, hence acyclic. Indeed, the restriction maps $\cF(U)((u))\to \cF(V)((u))$ are surjective since the maps $\cF(U)\to\cF(V)$ are surjective.

For arbitrary acyclic $\cG,$ take any of its flasque resolution $\cK^{\bullet}(\cG).$ Then $\cK^{\bullet}(\cG)((u))$ is a flasque resolution of $\cG((u)).$ Applying $\Gamma(X,-)$ to this resolution, we see that $\cG((u))$ is acyclic.\end{proof}

Now, we have a natural morphism of acyclic resolutions (which is of course a quasi-isomorphism):
$$G_{\leq d}(\cF((u)))\to G_{\leq d}(\cF)((u)).$$
By functoriality, it follows that we have a quasi-isomorphism of $\Z/2$-graded complexes of acyclic sheaves of $k((u))$-modules, with $u$-connections: $u$-connections:
$$(G_{\leq d}(\un{C}((u))),G_{\leq d}(\un{b})+uG_{\leq d}(\un{B}),G_{\leq d}(\un{\nabla_u}))\to (G_{\leq d}(\un{C})((u)),G_{\leq d}(\un{b})+uG_{\leq d}(\un{B}),G_{\leq d}(\un{\nabla_u})).$$
Applying to it the functor $\Gamma(X,-)$ and passing to cohomology, we obtain the desired isomorphism. This proves the proposition.
\end{proof}

\section{Affine case}
\label{sec:aff_case}

Fix some field $k$ of characteristic zero.

In this section we will prove Theorem \ref{intro:main_mixed} in the affine case.

Let $A$ be a smooth commutative algebra of finite type over $k,$ and take some $W\in A.$ Put $X:=\Spec A$
We assume that the critical locus of $W$ is contained in $W^{-1}(0)\subset X.$
Note that the  strict morphism of complexes with $u$-connections
\begin{equation}\label{eq:affcase1}(\Omega_A^{\bullet},-dW,d,\nabla_u^{DR})\to \bR\Gamma(X,(\Omega_X^{\bullet},-dW,d,\nabla_u^{DR}))\end{equation}
is a quasi-isomorphism. This follows from the fact that the sheaves $\Omega_X^p$ are acyclic since $X$ is affine.

So our goal is to construct a chain of weak quasi-isomorphisms between
$$\Hoch(MF_{coh}(X,W))\text{ and }(\Omega_A^{\bullet},-dW,d,\nabla_u^{DR}).$$

First, by Proposition \ref{prop:equiv_for_affine} and Proposition \ref{prop:q_is_for_q_eq} we have a strict quasi-isomorphism
\begin{equation}\Hoch((A,W)\text{-mod}^{cdg}_{fgp})\stackrel{\sim}{\to} \Hoch(MF_{coh}(X,W)).\end{equation}
Further, by Proposition \ref{prop:Pol-Pos} we have a strict quasi-isomorphism
\begin{equation}\Hoch((A,W)\text{-mod}^{cdg}_{fgp})\stackrel{\sim}{\to} \Hoch^{\Pi}((A,W)\text{-mod}^{cdg}_{fgp}).\end{equation}
Applying Proposition \ref{prop:q_is_for_ps_eq} to pseudo-equivalences from Proposition \ref{prop:cdg_qdg_ps_eq}, we obtain
strict quasi-isomorphisms
\begin{equation}\Hoch^{\Pi}((A,W)\text{-mod}^{cdg}_{fgp})\stackrel{\sim}{\to} \Hoch^{\Pi}((A,W)\text{-mod}^{qdg}_{fgp})\stackrel{\sim}{\leftarrow}\Hoch^{\Pi}(A,-W).\end{equation}
Next, we apply Proposition \ref{prop:opposite} to the curved DG algebra $(A,-W),$ and get a weak quasi-isomorphism
\begin{equation}\Hoch^{\Pi}(A,-W)\stackrel{\sim}{\to} \Hoch^{\Pi}(A,W).\end{equation}
Finally, by Proposition \ref{prop:HKR_map} we have a weak quasi-isomorphism
\begin{equation}\label{eq:affcase7}\Hoch^{\Pi}(A,W)\stackrel{\sim}{\to} (\Omega_A^{\bullet},-dW,d,\nabla_u^{DR}).\end{equation}

Combining quasi-isomorphisms \eqref{eq:affcase1}-\eqref{eq:affcase7}, we obtain the following result.

\begin{theo}\label{th:main_affcase}Let $X$ be a smooth affine scheme of finite type over a field $k$ of characteristic zero. Then there is a chain of weak quasi-isomorphisms
between mixed complexes with $u$-connections
$$\Hoch(MF_{coh}(X,W))\text{ and }(\Gamma(X,\Omega_X^{\bullet}),-dW,d,\nabla_u^{DR}).$$ \end{theo}

\section{General case}
\label{sec:gen_case}

Here we will prove Theorem \ref{intro:main_mixed} in the general case. Let $X$ be a smooth separated scheme of finite type over a field $k$ of characteristic zero, and $W\in\cO(X).$ Again, we assume
that the critical locus of $W$ is contained in $W^{-1}(0).$

Consider the following presheaves of CDG categories:
$$\un{MF_{coh}}(X,W)(U):=MF_{coh}(U,W);\quad \un{MF}^{nv}(X,W)(U):=MF^{nv}(U,W);$$
$$(\cO_X,W)\un{\text{-mod}}^{qdg}_{l.f.}(U):=(\cO_U,W)\text{-mod}^{qdg}_{l.f.};\quad (\cO_X,\pm W)(U):=(\cO_X(U),\pm W).$$
Here the CDG category $(\cO_U,W)\text{-mod}^{qdg}_{l.f.}$ is defined in the same way as $\cD\text{-mod}^{qdg}_{fgp},$ but we take locally free sheaves of
 finite rank on $U$ instead of finitely generated projective modules. Again, we have natural morphisms of presheaves of CDG categories
 $$\un{MF}^{nv}(X,W)\to (\cO_X,W)\un{\text{-mod}}^{qdg}_{l.f.}\leftarrow (\cO_X,-W).$$

Now our goal is to construct a chain of quasi-isomorphisms between
$$\Hoch(MF_{coh}(X,W))\text{ and }\bR\Gamma(X,(\Omega_X^{\bullet},-dW,d,\nabla_u^{DR})).$$

Combining Proposition \ref{prop:suffices_for_affine} with quasi-isomorphisms \eqref{eq:affcase1}-\eqref{eq:affcase7} from the previous section, we obtain strict quasi-isomorphisms of sheaves of mixed complexes with $u$-connections:
\begin{equation}\label{eq:gencase_1}\Hoch(\un{MF}^{nv}(X,W))\to \Hoch(\un{MF_{coh}}(X,W));\end{equation}
\begin{equation}\Hoch(\un{MF}^{nv}(X,W))\to \Hoch^{\Pi}(\un{MF}^{nv}(X,W));\end{equation}
\begin{equation}\Hoch^{\Pi}(\un{MF}^{nv}(X,W))\to \Hoch^{\Pi}((\cO_X,W)\un{\text{-mod}}^{qdg}_{l.f.})\leftarrow \Hoch^{\Pi}(\cO_X,-W);\end{equation}
\begin{equation}\Hoch^{\Pi}(\cO_X,-W)\to \Hoch^{\Pi}(\cO_X,W),\end{equation}
\begin{equation}\label{eq:gencase_7}\Hoch^{\Pi}(\cO_X,W)\to (\Omega_X^{\bullet},-dW,d,\nabla_u^{DR}).\end{equation}

Applying the functor $\bR\Gamma(X,-)$ to the quasi-isomorphisms \eqref{eq:gencase_1}-\eqref{eq:gencase_7}, and using Proposition \ref{prop:RGamma_pres_q_is} we get a chain of quasi-isomorphisms between
$$\bR\Gamma(X,\Hoch(\un{MF_{coh}}(X,W)))\text{ and }\bR\Gamma(X,(\Omega_X^{\bullet},-dW,d,\nabla_u^{DR})).$$
Thus, we are left to prove the following result

\begin{prop}\label{prop:map_to_RGamma}The morphism
\begin{equation}\label{map_to_RGamma}\Hoch(MF_{coh}(X,W))\to\bR\Gamma(X,\Hoch(\un{MF_{coh}}(X,W)))\end{equation}
is a quasi-isomorphism.\end{prop}

\begin{proof}The statement is just about Hochschild complexes without additional structures, so we can forget about Connes differential and the connection.

In \cite{Pos}, the support of a matrix factorization $(E,\delta)$ is defined as the smallest closed subset $\supp(E,\delta)\subset X$ such that the restriction
of $(E,\delta)$ onto $X\setminus \supp(E,\delta)$ is isomorphic to zero (in the homotopy category of matrix factorizations). Further, for any closed subset $Z\subset X$ denote by $MF_{coh,Z}(X,W)\subset MF_{coh}(X,W)$
the full DG subcategory consisting of matrix factorizations with support contained in $Z.$

By \cite{Pos}, for any closed subset $Z\subset X$ we have a short exact sequence of DG categories
$$MF_{coh,Z}(X,W)\to MF_{coh}(X,W)\to MF_{coh}(X\setminus Z,W).$$
It follows from Theorem \ref{th:Keller_exact_seq} that we have an exact triangle of Hochschild complexes:
\begin{multline}\label{eq:Keller_Hochschild}\Hoch(MF_{coh,Z}(X,W))\to \Hoch(MF_{coh}(X,W))\to \Hoch(MF_{coh}(X\setminus Z,W))\to\\ \Hoch(MF_{coh,Z}(X,W))[1].\end{multline}
Further, if $X=U_1\cup U_2$ is an open covering, then putting $Z=X\setminus U_1$ and applying triangle \eqref{eq:Keller_Hochschild} to $X$ and $U_2$
we get a pair of exact triangles

\begin{multline*}\Hoch(MF_{coh,Z}(X,W))\to \Hoch(MF_{coh}(X,W))\to \Hoch(MF_{coh}(U_1,W))\to\\ \Hoch(MF_{coh,Z}(X,W))[1].\end{multline*}
\begin{multline*}\Hoch(MF_{coh,Z}(U_2,W))\to \Hoch(MF_{coh}(U_2,W))\to \Hoch(MF_{coh}(U_1\cap U_2,W))\to\\ \Hoch(MF_{coh,Z}(U_2,W))[1].\end{multline*}

Since the natural restriction DG functor
$$MF_Z(X,W)\to MF_Z(U_2,W)$$ is a quasi-equivalence, we get a Mayer-Vietoris exact triangle
\begin{multline*}\Hoch(MF_{coh}(X,W))\to \Hoch(MF_{coh}(U_1,W))\oplus \Hoch(MF_{coh}(U_2,W))\to\\ \Hoch(MF_{coh}(U_1\cap U_2,W))\to \Hoch(MF_{coh}(X,W))[1].\end{multline*}
It naturally maps to the usual Mayer-Vietoris exact triangle for the complex of sheaves $\Hoch(\un{MF_{coh}}(X,W)):$
\begin{multline*}\bR\Gamma(X,\Hoch(\un{MF_{coh}}(X,W)))\to\\ \bR\Gamma(U_1,\Hoch(\un{MF_{coh}}(X,W)))\oplus \bR\Gamma(U_2,\Hoch(\un{MF_{coh}}(X,W)))\to\\ \bR\Gamma(U_1\cap U_2,\Hoch(\un{MF_{coh}}(X,W)))\to \bR\Gamma(X,\Hoch(\un{MF_{coh}}(X,W)))[1].\end{multline*}
Therefore, to show that \eqref{map_to_RGamma} is a quasi-isomorphism for the pair $(X,W),$ it suffices to show it for the pairs $(U_1,W),$ $(U_2,W)$ and $(U_1\cap U_2,W).$ But the scheme $X$ is separated of finite type, so the problem reduces to affine varieties. And in the affine case this follows from Theorem \ref{th:main_affcase}. This proves the proposition.
\end{proof}

We have proved the following theorem (Theorem \ref{intro:main_mixed}):

\begin{theo}\label{th:main_mixed}Under the above assumptions, we have a chain of quasi-isomorphisms between mixed complexes with $u$-connections:
$$(\Hoch(MF_{coh}(X,W)),b,B,\nabla_u^{GM}),\quad \bR\Gamma(X,(\Omega_X^{\bullet},-dW,d,\nabla_u^{DR})).$$\end{theo}

This Theorem implies implies the following (Theorem \ref{intro:main_algebraic})

\begin{theo}\label{th:main_algebraic}Under the above assumptions, we have natural isomorphisms
$$(HP_{\bullet}(MF_{coh}(X,W)),\nabla_u^{GM})\cong (H^{\bullet}(X,(\Omega_X^{\bullet}((u)),-dW+ud)),\nabla_u^{DR}),$$
and $$HH_{\bullet}(MF_{coh}(X,W))\cong H^{\bullet}(X,(\Omega_X^{\bullet},-dW)).$$\end{theo}

\begin{proof}Indeed, the second isomorphism directly follows from Theorem \ref{th:main_mixed}, and the first one is a combination of Theorem \ref{th:main_mixed} and Proposition \ref{prop:mixed_implies_k((u))}.\end{proof}

Further, we have the following corollary over complex numbers (Theorem \ref{intro:main_complex}).

\begin{theo}\label{th:main_complex}Let $X$ be a smooth quasi-projective algebraic variety over $\C.$ Then we have an identification of $\Z/2$-graded vector bundles with connection on the formal punctured disk:
$$(HP_{\bullet}(MF_{coh}(X,W)),\nabla_u^{GM})\cong \widehat{RH}^{-1}(H^{\bullet-1}_{an}(W^{-1}(0),\phi_W\C_X),T\cdot (-1)^{\bullet}).$$\end{theo}

\begin{proof}This follows from the first isomorphism of Theorem \ref{th:main_algebraic} and from Theorem \ref{th:Sabbah} (result of Sabbah).
Indeed, adding the term $\frac{\Gamma}{u}$ to the connection on the twisted de Rham complex corresponds to multiplying the monodromy by $(-1)^{\bullet}.$ This proves the theorem.\end{proof}

\section{Concluding remarks}
\label{sec:concluding}

Here we write down some remarks.

First, we know by Weibel \cite{W2} and Keller \cite{Ke2} that for any complex quasi-projective variety $X$ (not necessarily smooth) the periodic cyclic homology
$HP_{\bullet}(\Perf(X))$ is identified with $H^{\bullet}_{top}(X,\C).$

The naive hope would be that the similar result should hold for locally free matrix factorizations. More precisely, one can define the full DG subcategory $$MF_{l.f.}(X,W)\subset MF_{coh}(X,W)$$ to consist of matrix factorizations locally isomorphic (in the derived sense) to the
free matrix factorizations of finite rank. For smooth $X,$  these two categories coincide. However, it is shown in \cite{Pos} that the categories $MF_{l.f.}(X,W)$ in general behave very bad. In particular, the example discussed in \cite{Pos}, Section 3.3 shows that: Thomason localization theory fails; there may not exist a single generating object; periodic cyclic homology can be infinite-dimensional. Thus, we cannot expect any kind of analogue of Theorem \ref{intro:main_complex} for the categories $MF_{l.f.}(X,W)$ for singular $X.$

On the other hand, the categories $MF_{coh}(X,W)$ behave very well for arbitrary separated $X$ of finite type over a field of characteristic zero. We expect an isomorphism
$$(HP_{\bullet}(MF_{coh}(X,W)),\nabla_u)\cong \widehat{RH}^{-1}((H^{\bullet-1}_{c}(W^{-1}(0),\phi_W\C_X))^{\vee},T^{\vee}\cdot (-1)^{\bullet}).$$

Of course, these isomorphisms should admit a version with support: for any closed subscheme $Z\subset X,$ we can take matrix factorizations with support on $Z$ in the LHS, and the cohomology supported on $Z$ in the RHS.

The $\Z$-grading on the vanishing cohomology should come from lambda-decomposition for the sheaf of Hoschschild complexes of the sheaf of curved DG algebras
$(\cO_X,W),$ as in \cite{W2} for $W=0.$ In our case of smooth varieties this is straightforward.

We also note that Thom-Sebastiani theorem for vanishing cohomology \cite{M} over complex numbers can be viewed as a combination of the K\"unneth isomorphism for peridoic cyclic homology (\cite{HS}, \cite{Sh2}) for matrix factorizations categories,
and the Thom-Sebastiani Theorem for matrix factorizations \cite{Pr}.

Also, in the case of compact critical locus and smooth $X,$ the Kontsevich-Soibelman degeneration conjecture  for the DG category $MF_{coh}(X,W)$ is a combination of our Theorem \ref{intro:main_mixed} and the theorem of Barannikob-Kontsevich (unpublished) on degeneration for the twisted de Rham complex, proved also in \cite{Sab1} and \cite{OV}.

Finally, for a smooth algebraic variety over $\C,$ and a regular function $W,$ the categorical Chern character gives a map
$$K_0(MF_{coh}(X,W)^c)\to H^{odd}_{an}(W^{-1}(0),\phi_W\C_X)^T$$
(here we get odd cohomology because in the case of $W=0$ the standard degree on cohomology is shifted by $1.$). The superscript $^c$
means that we take the Karoubi completion. This map can be shown to map to Hodge classes. The analogue of Hodge conjecture (see also \cite{KKP})
states that this Chern character $\otimes\Q$ is surjective onto rational Hodge classes (at least in the case of compact critical locus).

\end{document}